\documentclass[11pt,a4paper]{article}
\usepackage[utf8]{inputenc}
\usepackage{amsmath,amsfonts,amssymb,amsthm}
\usepackage{mathtools}
\mathtoolsset{showonlyrefs}
\newtagform{lhs}{(}{)}
\usepackage{xcolor}
\usepackage[colorlinks=true,
    linkcolor=blue,
    filecolor=magenta,      
    urlcolor=cyan]{hyperref}
\usepackage[left=3cm,right=3cm,top=2cm,bottom=2.5cm]{geometry}
\usepackage{dsfont}

\usepackage[square, numbers, sort&compress]{natbib}
\usepackage{doi}
\newcommand{\EE}{\mathbb{E}}

\newcommand{\R}{\mathbb{R}}
\newcommand{\IP}{\mathbb{P}}

\newcommand{\E}{\mathbb{E}}

\newcommand{\cA}{\mathcal{A}}

\newcommand{\cG}{\mathcal{G}}

\newcommand{\cF}{\mathcal{F}}

\newcommand{\cM}{\mathcal{M}}

\newcommand{\1}{\mathds{1}}

\newcommand{\orm}{\mathrm{o}}

\newcommand{\eps}{\varepsilon}

\newcommand{\floor}[1]{\lfloor #1 \rfloor}

\newcommand{\Var}{\text{Var}}

\newcounter{thm}
\numberwithin{thm}{section}
\numberwithin{equation}{section}
\newcounter{alpha}

\newtheorem{theorem}[thm]{Theorem}
\newtheorem{coro}[thm]{Corollary}
\newtheorem{lem}[thm]{Lemma}
\newtheorem{prop}[thm]{Proposition}
\theoremstyle{definition}

\newtheorem{defi}[thm]{Definition}
\newtheorem{rem}[thm]{Remark}
\newtheorem*{examples}{Examples}

\title{\bf A universal scaling limit for diffusive amnesic step-reinforced random walks}
\author{Marco Bertenghi\thanks{Institute of mathematics, Winterthurerstrasse 190, 8057 Zürich, and University of Zürich, Zürich, Switzerland \hfill\texttt{marco.bertenghi@math.uzh.ch}}\; 
and Lucile Laulin\thanks{Modal’X - UMR CNRS 9023, UPL, Université Paris Nanterre, 92000 Nanterre, France, and FP2M, CNRS FR 2036 \hfill \texttt{lucile.laulin@math.cnrs.fr}}
}
\date{}

\begin{document}

\maketitle

\begin{abstract}
We introduce a variation of the step-reinforced random walk with general memory. For the diffusive regime, we establish a functional invariance principle and show that, given suitable conditions on the memory sequence, the arising limiting processes are always the sum of a noise reinforced Brownian motion and a (not independent) Brownian motion. 
\end{abstract}

\section{Introduction}

Motivated by the study of the effects of memory on the asymptotic behaviour of non-Markovian processes, Schütz and Trimper \cite{SchuetzTrimper} introduced around 20 years ago the so-called \textit{elephant random walk}. The elephant random walk can be understood as a fundamental example of a step-reinforced random walk. Additionally, it stands as one of the simplest models that lead to anomalous diffusion.

Anomalous diffusion appears in many physical, biological or social systems such as  human travel \cite{Brockmann} or heartbeat intervals and DNA sequences \cite{Buldy}. Further examples include telomeres in the nucleus of cells \cite{Bronstein}, ion channels in the plasma membrane \cite{Weigel}, diffusion in porous materials \cite{Koch}, or diffusion in polymer networks \cite{Wong} to name only a few. The phenomena of anomalous diffusion often arises in theoretical models by incorporating memory effects such as modeled by the elephant random walk.

The elephant random walk is a discrete-time nearest neighbour random walk on the integer lattice $\mathbb{Z}$ with infinite memory, in allusion to the traditional saying that \textit{an elephant never forgets}. The dynamics of the elephant random walk are governed by a parameter $p$ between zero and one, commonly referred to as the \textit{memory parameter}, that specifies the probability of repetition of certain steps. Roughly speaking, given an initial step of the elephant, say $X_1=1$ a.s., then at each integer time $n \geq 2$, the elephant remembers one of its previous steps chosen uniformly at random; then it decides, either with probability $p$ to repeat this step, or with complementary probability of $1-p$ to walk in the opposite direction. Notably, the steps of the elephant are either plus or minus one. As a consequence of the aforementioned dynamics, for $p > 1/2$, the elephant is more inclined to continue its walk in the average direction it has already taken up to that point. Conversely, for $p<1/2$, it tends to backtrack. In the borderline case of $p=1/2$, the elephant does not intend to make a decision and its path follows that of a simple symmetric random walk on $\mathbb{Z}$. In particular, the elephant random walk is a time-inhomogeneous Markov chain, although some works in the literature improperly assert its non-Markovian character. Indeed, if the elephant is at position $k \in \mathbb{Z}$ at time $n \in \mathbb{N}$, then it performed $(n+k)/2$ steps up and $(n-k)/2$ steps down, more information from the past is irrelevant for predicting the $(n+1)$th step. The asymptotic behaviour after a proper rescaling of the elephant random walk has in recent years been a topic of interest for many authors and is well understood, see \cite{BaurBertoin, Bercu, GavaSchuetzColetti, Coletti2017, GuevaraERW, KubotaTakei} and \cite{Bertenghi, BercuLaulinCenter, BaurClass, LaulinReinforcedERW, GutERWGradIncrMemory, BercuMRW, GuevaraMinimal, GonzalesMult, TakeiMiyazaki} for variations. 

A step-reinforced random walk extends the dynamics of the elephant random walk to allow for more diverse steps, rather than restricting to plus or minus one.  In essence, the steps can follow an arbitrary distribution, typically on $\mathbb{R}$. Put simply, we are given a sequence $\mathcal{X}_1, \mathcal{X}_2, \dots$ of i.i.d. copies of a random variable $\mathcal{X}$ on $\mathbb{R}$ and again a memory parameter $p$ between zero and one. We then create a sequence of step-reinforced random variables $X_1,X_2, \dots$ as follows: $X_1= \mathcal{X}_1$ a.s.and subsequently, at each integer time $n\geq2$, one of the previous steps is chosen uniformly at random. Then, with probability $p$ the step is repeated; otherwise, with the complementary probability of $1-p$, an independent increment following the same distribution as $\mathcal{X}$ is taken. In this setting, when $\mathcal{X}$ follows a Rademacher distribution, Kürsten \cite{Kuersten} (see also \cite{Gonzales}) pointed out that the step-reinforced random walk is then just a version of the elephant random walk with memory parameter $q=(p+1)/2 \in (1/2,1)$ in the present notation. Observe that in the degenerate case $p=1$, the dynamics of the step-reinforced random walk become essentially deterministic. Indeed, when $p=1$, then the position of the step-reinforced random walk at time $n$ is just given by $n \mathcal{X}_1$ for all $n \geq 1$, in particular the only remaining randomness for this process stems from the random variable $\mathcal{X}_1$. Similarly, when $p=0$, then the step-reinforced random walk reduces to a random walk with i.i.d. increments. In light of this, we will exclude these degenerate cases in our analysis, that is we will only consider $p \in (0,1)$.

In this work,  we introduce an additional layer of complexity to the step-reinforced random walk model. Specifically, we consider a more general underlying memory mechanism that incorporates recent steps being repeated with a higher likelihood, inspired by diminishing memory effects like amnesia. This dynamic will be governed by yet another parameter $\alpha \in \mathbb{R}_+$ that we shall henceforth refer to as the \textit{amnesia parameter}. In that direction for $\alpha \gg 1$, our amnesic step-reinforced random walk is much more likely to repeat steps from its recent past, whereas for $\alpha$ close (or equal) to $1$ it behaves more (or exactly) like a step-reinforced random walk. Laulin in \cite{LaulinAmnesia} considered a version of the elephant random walk where the underlying memory process also includes amnesia, see also \cite{chenlaulin2023} for the multidimensional extension. In contrast to the aforementioned work, our study encompasses a broader range of underlying memory distributions, ultimately including the one detailed in \cite{LaulinAmnesia}. We contend that the memory distribution employed here is both comprehensive and representative, encapsulating all known cases from the literature as specific instances. Furthermore, it remains sufficiently tractable for a thorough analysis.

It is our main objective in this work to establish a functional invariance principle for the properly rescaled amnesic step-reinforced random walk in in the so-called diffusive regime. We will show that the resulting limiting processes comprise the non-independent sum of a noise-reinforced Browian motion and a Brownian motion. For the properly re-scaled (non-amnesic) step-reinforced random walk Bertoin in \cite{BertoinUniversality} established the noise-reinforced Brownian motion as the universal scaling limit in the diffusive regime. Our work corroborates this result and further indicates the presence of a Brownian motion in the limiting process for all $\alpha \neq 1$, which also agrees with Theorem 2.3 in \cite{LaulinAmnesia}. A noise-reinforced Brownian motion is a simple real-valued and centered Gaussian process $\hat{B}=( \hat{B}(t))_{ \geq 0}$ with covariance function given by 

\begin{align*}
    \EE \left( \hat{B}(t) \hat{B}(s) \right) = \frac{1}{1-2p} s \left( \frac{t}{s} \right)^p \quad \text{ for } 0 \leq s \leq t \quad \text{ and } p \in (0,1/2).
\end{align*}
This process has notably appeared as the scaling limit for diffusive regimes of the ERW and other Pólya urn related processes, see \cite{BaurBertoin, GavaSchuetzColetti, Bertenghi, BaiGaussian}.

The remainder of the paper is organised as follows: In Section \ref{sec:Model} we will give an exact definition of the amnesic step-reinforced random walk. In Section \ref{sec:MainResults} we will present the main results of our work. Section \ref{sec:RegVarSeq} contains a short detour to regularly varying sequences, which are quintessential in the definition (and therefore the analysis) of amnesic step-reinforced random walks. In Section \ref{sec:TwoDimMartingale} we lay the ground work of our analysis before presenting the proofs of our main results in Section \ref{sec:Proofs}. For the readers convenience, and to make this work self-contained, technical lemmas and a non-standard result on martingales are provided in the appendix.

\section{The Model} \label{sec:Model}

In this section, we formally introduce our model. We start by giving our memory sequence, which is the main character of our work. Consider a positive sequence $(\mu_n)$ and define the sequence $(\nu_n)$ as follows:
\begin{equation*}
    \nu_0 = 0,\quad \nu_{n}=\nu_{n-1} + \mu_{n} \; \text{for $n\geq1$}.
\end{equation*}
The sole (yet crucial) additional assumption we make regarding the sequence $(\mu_n)$ is the following:
\begin{itemize}
    \item[\textbf{(A)}] the sequence $(\mu_n)$ is regularly varying (at infinity) with index $\beta \geq 0$.
\end{itemize} 
This ensures that the sequence $(\nu_n)$ is also regularly varying of {index $\alpha=\beta+1 \geq 1$. Roughly, this means that this sequences will have a polynomial growth and the memory process tends to prioritize recent times over older ones.
In Section \ref{sec:RegVarSeq} we provide more details and references on the topic of regularly varying sequences.

Now, our memory sequence $( \beta_n : n \geq 2)$ is distributed as follows: 
\begin{align}
    \IP( \beta_{n} = k) = \frac{\mu_k}{\sum_{i=1}^n \mu_i} = \frac{\nu_{k}-\nu_{k-1}}{\nu_n}, \qquad \text{for } 1 \leq k \leq n.
\end{align}
Hereafter, consider a sequence $\mathcal{X}_1,\mathcal{X}_2, \dots$ of i.i.d. copies of a random variables $\mathcal{X}$ on $\mathbb{R}$ with finite second order moment.  We define $X_1,  X_2, \dots$ recursively as follows: Let $( \varepsilon_n : n \geq 2)$ be an independent sequence of Bernoulli random variables with parameter $p \in (0,1)$, also independent of $(\beta_n)$.
Initially, set $X_1= \mathcal{X}_1$, and next for $n \geq 1$, define 
\begin{align}
\label{eq:startRW}
{X}_{n+1} = 
\begin{cases} { \mathcal{X} }_{n+1}, & \text{if } \varepsilon_{n+1}=0, \\
X_{\beta_{n}} , & \text{if } \varepsilon_{n+1}=1. 
\end{cases}
\end{align}
Finally, the sequence 
\begin{align}
    S_n = X_1 + \dots + X_n
\end{align}
is referred to as a step-reinforced random walk. 
The definition of the sequence $(X_n)$ implies that for any bounded and measurable $f: \mathbb{R} \to \mathbb{R}_+$
\begin{align} \label{eq:RobustnessSRRW}
   \EE( f(X_{n+1})) = (1- p) \EE(f(\mathcal{X }_{n+1})) + \frac{p}{\nu_n} \sum_{k=1}^n \mu_k \EE(f(X_k))
\end{align}
and it follows by induction that each $X_n$ also has law $\mathcal{X}$. 

Finally, for the rest of the paper, we will assume that we are in the diffusive regime which corresponds in our case to
$0<p < \frac{2 \alpha -1}{2 \alpha}$. This condition appears when we study the quadratic variations of our martingales from Section \ref{sec:TwoDimMartingale} as we want the variations to be regularly varying of positive index.

In precise terms, our work holds true for $\alpha>1/2$. However, we chose to focus on the case of $\alpha \geq 1$ to underscore the presence of amnesia. Indeed, the tendency shifts for $1/2<\alpha<1$, as the process exhibits a preference for moments from the early times of the past.

\section{Main results} \label{sec:MainResults}

In this section we introduce our main results, Theorem \ref{thm:LLN} and Theorem \ref{thm:MainResult}.

\subsection{Law of large numbers}

The (strong) law of large numbers will be an essential tool in establishing the functional invariance principle for the amnesic step-reinforced random walk. As such, Section \ref{sec:Proofs} will first focus on establishing a proof of Theorem \ref{thm:LLN}.

\begin{theorem}[Strong law of large numbers] \label{thm:LLN} For $ \frac{\alpha-1}{\alpha} <p < \frac{2 \alpha-1}{2 \alpha}$ we have the almost sure convergence
\begin{align}
  \lim_{n \to \infty} \frac{S_n}{n} = \E(\mathcal{X}).
\end{align}
\end{theorem}

We remark that Theorem \ref{thm:LLN} is trivially true for $p=0$. Indeed, in said case the step-reinforced random walk is just a sequence of centered i.i.d. increments and therefore the (strong) law of large number applies. 

In the setting under which we are in this work, it is sufficient to have the above LLN for $p<\frac{2\alpha-1}{2\alpha}$. However, we strongly believe that this holds for the full range $0<p<1$ (as long as $\alpha>0$). Proving this convergence goes beyond the purpose of this paper, hence we provide the proof in the diffuse regime only.

\subsection{A functional invariance principle}
The main result of this exposition is the following statement:

\begin{theorem} \label{thm:MainResult}
    Suppose that $ \frac{\alpha-1}{\alpha}< p < \frac{2 \alpha-1}{2\alpha}$, then we have the following convergence in distribution in $D([0, \infty))$ as $n$ tends to infinity 
    \begin{align}
        \left( \frac{S_{\lfloor nt \rfloor}-nt\E(\mathcal{X})}{\sqrt{n\Var(\mathcal{X})}}, t \geq 0 \right) \implies (W_t, t \geq 0)
    \end{align}
    where $(W_t, t \geq 0)$ is a real-valued, continuous and centered Gaussian process starting from the origin with covariance given for $0 \leq s \leq t$ by 
\begin{align}
    \E\left(W_sW_t\right) &=\frac{\alpha-1}{(1-p)(\alpha(1-p)-1)}s
    +\frac{p(\alpha(2-p)-1)}{(1-p)(2\alpha(1-p)-1)(1-\alpha(1-p))}s\Big(\frac{t}{s}\Big)^{1-\alpha(1-p)}.
\end{align}
\end{theorem}
Observe that for $\alpha=1$, Theorem \ref{thm:MainResult} recovers Theorem 3.3 in \cite{BertoinUniversality}. Moreover, we notice that our scaling coefficients in the covariance of $(W_t, t \geq 0)$ agree with Display (2.4) of Theorem 2.3 in \cite{LaulinAmnesia}. Furthermore, we notice that as the amnesia parameter $\alpha$ increases, the memory parameter $p$ also tends towards one which means that reinforcement of the steps is more likely to occur. However, as $\alpha \nearrow \infty$ we have $p=1$ and in said case Theorem \ref{thm:MainResult} clearly does not apply. Furthermore, as $\alpha$ gets close to $1/2$, we notice that this forces the memory parameter $p$ to be close to zero.

Let us define the coefficients in the covariance function in Theorem \ref{thm:MainResult} for fixed but arbitrary $\frac{\alpha-1}{\alpha}<p< \frac{2 \alpha-1}{2\alpha}$ as $c_1(p, \alpha):=c_1(\alpha)$ respectively $c_2(p, \alpha):=c_2(\alpha)$. 

\begin{table}[h]
\centering
\begin{tabular}{|c|c|c|c|}
\hline
$\alpha \in$                  & $(1/2,1)$ & $\{1\}$ & $(1,\infty)$ \\ \hline
$\text{sgn}(c_1(\alpha))$ & $+$         & $0$   & $-$           \\ \hline
$\text{sgn}(c_2(\alpha))$ & $+$         & $+$   & $+$            \\ \hline
\end{tabular}
\caption{Distribution of the signs of the coefficients $c_1,c_2$ with respect to the amnesia parameter $\alpha$.}
\label{tab:signs}
\end{table}
\begin{rem}
The covariance's structure of the process $(W_t, t \geq 0)$ looks a lot like the covariance structure of the sum of two independent Gaussian processes: a Brownian motion with scaling coefficient $\sqrt{c_1(\alpha)}$ and a noise reinforced Brownian motion with reinforcement $1-\alpha(1-p)$ and a scaling coefficient $\sqrt{{(2\alpha(1-p)-1)}c_2(\alpha)}$. However,
    as we will see in the proof of the Theorem, the process $(W_t, t \geq 0)$ is indeed the sum of such two processes, but those are in fact not independent. The two "scaling quantities" can be negative (see Table \ref{tab:signs}) and thus cannot be scaling Brownian motions.
\end{rem}

\section{Regularly varying sequences} \label{sec:RegVarSeq}
To make this exposition self-contained, we recall here some useful results
from the theory of regularly varying sequences and functions, see \cite{Bojanic73,Galambos73} for more details on this subject.

\begin{defi}[Slowly varying function] \label{def:SlowlyVarying} A function $L : (0, + \infty) \to (0, + \infty)$ is called \textit{slowly varying} (at infinity) if 
\begin{align}
     \frac{L(tx)}{L(x)} \xrightarrow[x \to \infty]{} 1, \quad \text{for all } t >0.
\end{align}
\end{defi}
Intuitively the meaning of Definition \ref{def:SlowlyVarying} is that the growth rate of the function $L$ does not change drastically as its input becomes large. 

\begin{defi}[Regularly varying function] \label{def:RegularlyVarying}
A function $R: (0, + \infty) \to (0, + \infty)$ is called {\itshape regularly varying} of index $\alpha \in \mathbb{R}$ if there exists a slowly varying function $L$ such that
\begin{align}
    R(x)=x^\alpha L(x).
\end{align}
\end{defi}

Observe that Definition \ref{def:RegularlyVarying} entails that the index $\alpha$ determines how the function behaves at infinity. Indeed, if $\alpha =0$, then $R$ is (always) slowly varying. For $x$ large enough, if $\alpha >0$, then $R$ roughly behaves like an increasing function, whereas if $\alpha <0$, then $R$ roughly behaves like a decreasing function. Moreover, Definition \ref{def:RegularlyVarying} immediately yields that $R$ is regularly varying of index $\alpha \in \mathbb{R}$ if and only if 
\begin{align}
    \frac{R(tx)}{R(x)} \xrightarrow{x \to \infty} t^\alpha, \quad \text{for all } t > 0.
\end{align}

\begin{defi}[Regularly varying sequence] \label{def:RegVarSequence}
A sequence of positive terms $(u_n)$ is called {\itshape regularly varying} of index $\alpha \in \mathbb{R}$ if there exists a regularly varying function $R$ of index $\alpha$ such that $u_n=R(n)$.
\end{defi}
We next give two examples of regularly varying sequences related to our work.
\begin{examples}
\begin{enumerate}
    \item Let $\mu_n = 1$, then $(\mu_n)$ is regularly varying of index $0$ and $\nu_n=n$ so that $(\nu_n)$ is regularly varying of index $\alpha=1$. This setting corresponds to the classical ERW.
    \item Let $\mu_n=\frac{\Gamma(n+\delta)}{\Gamma(n) \Gamma(\delta+1)}$, then $(\mu_n)$ is regularly varying of index $\delta$ and $(\nu_n)$ is regularly varying of index $\alpha=\delta+1$. This is exactly the memory introduced by Laulin, see Display (1.3) in \cite{LaulinAmnesia}.
\end{enumerate}
\end{examples}

During our computations, we will often make use of the following results:

\begin{theorem}[See Display (1.1) and Theorem 6 in \cite{Bojanic73} ]\label{thm:RegVarCharacterisation} A sequence of positive numbers $(u_n)$ is regularly varying of index $\alpha > -1$ if and only if 
\begin{align}
    \frac{1}{n u_n}\sum_{k=1}^n u_k \underset{n\to\infty}\longrightarrow \frac{1}{1+\alpha}.
\end{align}
\end{theorem}

\begin{theorem}[Theorem 4 in \cite{Bojanic73}] \label{theorem:RVSCharacterisation}
    A sequence $(u_n)$ of positive numbers is a regularly varying
sequence of index $\alpha$ if and only if there is a sequence of positive numbers
$(v_n)$ such that $ u_n\sim v_n$ and
\begin{align} \label{condition:SlowlyVaryingSeq}
    \lim_{n\to\infty} n\Big(1-\frac{v_{n-1}}{v_n}\Big) = \alpha.
\end{align}
\end{theorem}
In particular, the sequences $(v_n)$ which satisfy condition \eqref{condition:SlowlyVaryingSeq} are regularly varying sequences of index $\alpha$. 
A direct consequence of Theorem \ref{thm:RegVarCharacterisation} is that
\begin{coro}
\label{cor:charac-rv}
    If $(u_n)$ is regularly varying of index $\alpha$ then,
\begin{equation}
    \frac{u_{n+1}}{u_n} = 1 +\frac{\alpha}{n} + o\Big(\frac{1}{n}\Big).
\end{equation}
\end{coro}
Further, see condition (B) in \cite{Galambos73}, we have the following equivalence.
\begin{prop}
\label{prop:equiv-RV}
    A sequence $(u_n)$ of positive numbers is regularly varying of index $\alpha$ if and only if $(n^{-\sigma}u_n)$ is eventually increasing for each $\sigma < \alpha$ and $(n^{- \tau} u_n)$ is eventually decreasing for every $\tau > \alpha$.
\end{prop}

\section{A two dimensional martingale approach} \label{sec:TwoDimMartingale}
This section is a preliminary section to prove the main theorems. We will present two martingales that will be crucial for our analysis. To do this, we first assume, without the loss of generality, that $\mathcal{X}$ is centred and normalised, i.e. $\E( \mathcal{X} )=0$ and $\sigma^2 = \Var( \mathcal{X} )=1$. For what follows, we shall make these two assumptions implicitly.
\begin{lem} \label{lemma:Martingales} For $n \geq 1$, we define the following deterministic sequences 
\begin{align}
    \gamma_n = 1+ p\frac{\mu_{n+1}}{\nu_n}, \quad 
    a_n = \prod_{k=1}^{n-1}\gamma_k^{-1}, \quad
    \eta_n = \sum_{k=1}^{n-1} \frac{1}{a_k \nu_k}.
\end{align}
Further, we set $Y_n = \sum_{k=1}^n \mu_k X_k$. Then $(M_n)$ and $(N_n)$ defined by
\begin{equation}
    M_n = a_n Y_n \quad\text{and}\quad N_n = S_n - p \eta_n M_n
\end{equation}
are square-integrable martingales. 
\end{lem}
\begin{proof} Since $\E(X^2_k) < \infty$ for all $k \in \mathbb{N}$, the square integrability of $M_n$ and $N_n$ is immediate.

Further, by \eqref{eq:startRW} it follows that 
\begin{align}
    \EE( X_{n+1} \mid \cF_n) &= p \EE(X_{\beta_{n+1}} \mid \cF_n)  \\
    &= p \EE \left( \sum_{k=1}^n X_k \mathbf{1}_{\beta_{n+1}=k} \mid \cF_n \right) \\
    &= \frac{p}{ \nu_n }  \sum_{k=1}^n \mu_k X_k  \\
    &= \frac{p}{ \nu_n} Y_n.  \label{eq:step1}
\end{align}
In turn, \eqref{eq:step1} yields 
\begin{align}
    \EE(Y_{n+1} \mid \cF_n) &= Y_n + \mu_{n+1} \EE(X_{n+1} \mid \cF_n) \\
    & = \left( 1 + p\frac{\mu_{n+1}}{\nu_n} \right) Y_n \\
    & = \gamma_n Y_n. \label{eq:FirstMartingale}
\end{align}
From \eqref{eq:FirstMartingale} it is then immediate that $M_n= a_n Y_n$ is a martingale.
Furthermore, \begin{align}
    \EE(N_{n+1} \mid \cF_n) &= S_n + \EE(X_{n+1} \mid \cF_n) -p \eta_{n+1} M_n \\
    &= S_n + p \left( \frac{Y_n}{\nu_n}- \eta_{n+1}M_n \right) \\
    &= S_n + p \left( \frac{1}{\nu_n a_n} - \eta_{n+1} \right) M_n \\
    &= S_n - p \eta_n M_n = N_n
    \label{eq:SecondMartingale}
\end{align}
and \eqref{eq:SecondMartingale} entails that $(N_n)_{n \geq 0}$ is also a martingale. 
\end{proof}

Observe that Lemma \ref{lemma:Martingales} allows us to rewrite $S_n$ as 
\begin{align} \label{eq:MartingaleIdentity}
    S_n = N_n + p \eta_n M_n
\end{align}
and equation \eqref{eq:MartingaleIdentity} allows us to establish the asymptotic behaviour of $S_n$ via an extensive use of martingale theory. In order to investigate the asymptotic behaviour of $(S_n)$ via \eqref{eq:MartingaleIdentity}, we introduce the two-dimensional martingale $( \mathcal{M}_n)$ defined by 
\begin{align}
    \mathcal{M}_n = \binom{N_n}{M_n}
\end{align}
where $(M_n)$ and $(N_n)$ are the two square-integrable martingales introduced in Lemma \ref{lemma:Martingales}.

\begin{lem} \label{lemma:QuadraticVariationProcess} The quadratic variation of $( \cM_n)$ is given by 
\begin{align}
    \langle \mathcal{M} \rangle_n & = \sum_{k=0}^{n-1} \left( (1-p) + \frac{p}{\nu_k} U_k \right) \begin{pmatrix} \left( 1-p a_{k+1}\eta_{k+1}\mu_{k+1}\right)^2 & a_{k+1}\mu_{k+1}-p a_{k+1}^2 \eta_{k+1}\mu_{k+1}^2 \\
    a_{k+1}\mu_{k+1}- p a_{k+1}^2 \eta_{k+1}\mu_{k+1}^2 & a_{k+1}^2 \mu_{k+1}^2
    \end{pmatrix} \notag \\  &  \quad\quad\quad- \xi_n,  \label{eq:QuadraticVariationProcess}
\end{align}
where $U_n = \sum_{k=1}^n \mu_k X_k^2$ and 
\begin{align}
    \xi_n = \sum_{k=0}^{n-1} \frac{p^2}{\nu_k^2}Y_k^2 \begin{pmatrix} \left( 1-p a_{k+1}\eta_{k+1}\mu_{k+1}\right)^2 & a_{k+1}\mu_{k+1}-pa_{k+1}^2 \eta_{k+1}\mu_{k+1}^2 \\
    a_{k+1}\mu_{k+1}-p a_{k+1}^2 \eta_{k+1}\mu_{k+1}^2 & a_{k+1}^2 \mu_{k+1}^2.
    \end{pmatrix}
\end{align}
\end{lem}

\begin{proof}
    Denote the martingale increment $t_{n+1}=X_{n+1}-\frac{p}{\nu_n}Y_n$  and observe that indeed 
\begin{align} 
\label{cond:Duflo1}
    \EE( t_{n+1} \mid \mathcal{F}_n) = 0.
\end{align}
We obtain 
\begin{align} 
    \Delta \mathcal{M}_{n+1} &= \mathcal{M}_{n+1}- \mathcal{M}_n \\
    &= \binom{S_{n+1}-S_n -p ( \eta_{n+1} M_{n+1}- \eta_n M_n)}{a_{n+1}Y_{n+1}-a_nY_n} \\
    & = \binom{X_{n+1}-p ( \eta_{n+1}a_{n+1}Y_{n+1}- \eta_n a_n Y_n)}{ a_{n+1}\mu_{n+1} t_{n+1}} \\
    &= \binom{(1-p a_{n+1}\eta_{n+1}\mu_{n+1})t_{n+1}}{a_{n+1}\mu_{n+1}t_{n+1}} \\
    &= \binom{1-p a_{n+1}\eta_{n+1}\mu_{n+1}}{a_{n+1}\mu_{n+1}} t_{n+1}.
    \label{eq:DeltaM}
\end{align}
Further we have
\begin{align}
\label{eq-bound-epsilon}
    \EE \left( t_{n+1}^2 \mid \mathcal{F}_n \right) &= \EE(X_{n+1}^2 \mid \cF_n)- \frac{p^2}{\nu_n^2}Y_n^2 \notag \\
    &= (1-p) +  \frac{p}{\nu_n} \sum_{k=1}^n \mu_k X_k^2 - \frac{p}{\nu_n^2}Y_n^2 \notag \\
    &= (1-p) + \frac{p}{\nu_n} U_n - \frac{p}{\nu_n^2}Y_n^2.
\end{align}
In turn, this yields
\begin{align}
    \EE\left( (\Delta \mathcal{M}_{n+1}) ( \Delta \mathcal{M}_{n+1})^T \mid \mathcal{F}_n \right) &= \\
      \left((1-p) +  \frac{p}{\nu_n} U_n - \frac{p}{\nu_n^2}Y_n^2\right) &\begin{pmatrix} \left( 1-p a_{n+1}\eta_{n+1}\mu_{n+1}\right)^2 & a_{n+1}\mu_{n+1}-aa_{n+1}^2 \eta_{n+1}\mu_{n+1}^2 \\
    a_{n+1}\mu_{n+1}-aa_{n+1}^2 \eta_{n+1}\mu_{n+1}^2 & a_{n+1}^2 \mu_{n+1}^2
    \end{pmatrix}  \label{eq:StepToQV}.
\end{align}
Thanks to \eqref{eq:StepToQV} we immediately arrive at \eqref{eq:QuadraticVariationProcess}.
\end{proof}
Then, we find that:
\begin{coro} \label{corollary:QuadVarMnNn}
 We have 
 \begin{align} \label{eq:QuadraticVar1}
    \langle M \rangle_n &=  \sum_{k=1}^n \left( (1-p) + \frac{p}{\nu_k} U_k \right) (a_k \mu_k)^2 - \zeta_n, \\  & \text{where} \qquad \zeta_n = p^2\sum_{k=1}^{n} \frac{a_k^2}{\nu_{k-1}^2}\mu_k^2 Y_{k-1}^2 
\end{align}
and
\begin{align}
    \langle N \rangle_n &= \sum_{k=1}^n \left( (1-p) + \frac{p}{\nu_k} U_k \right)\left( 1-p a_{k}\eta_{k}\mu_k\right)^2 - \chi_n \\ 
    & \text{where} \qquad  \qquad \chi_n =  \sum_{k=1}^{n} \frac{p^2}{\nu_{k-1}^2} \left( 1-p a_{k}\eta_{k}\mu_k\right)^2 Y_{k-1}^2.
\end{align}
\end{coro}

The asymptotic behaviour of $M_n$ is closely related to the one of 
\begin{align}
\label{eq:def-tilde-w}
    \tilde{w}_n = \sum_{k=1}^n \left( (1-p) + \frac{p}{\nu_k} U_k \right) (a_k \mu_k)^2
\end{align}
as one can observe from \eqref{eq:QuadraticVar1} that we always have $\langle M \rangle_n \leq w_n$ and that $\zeta_n$ is negligible when compared to $w_n$. By the same token, the asymptotic behavior of $N_n$ is closely related to the one of $z_n$ where
\begin{equation}
\label{eq:def-tilde-z}
   \tilde{z}_n = \sum_{k=1}^n \left( (1-p) + \frac{p}{\nu_k} U_k \right)\left( 1-p a_{k}\eta_{k}\mu_k\right)^2.
\end{equation}

\section{Proof of main results} \label{sec:Proofs}
In this section we give detailed proofs of our main Theorems. 

We first need to establish a proof of the law of large numbers (Theorem \ref{thm:LLN}) as this result is a requirement in order to establish a proof of our main result (Theorem \ref{thm:MainResult}).

\subsection{Proof of Theorem \ref{thm:LLN}}
We now give the proof of the strong law of large numbers (Theorem \ref{thm:LLN}).
\begin{proof}[Proof of Theorem \ref{thm:LLN}]
Recall from \eqref{eq:MartingaleIdentity} that we have the decomposition
\begin{align}
    N_n = S_n  + p \eta_n M_n.
\end{align}
It follows from Corollary \ref{corollary:QuadVarMnNn} together with Corollary \ref{corollary.AsymptoticRates} that almost surely
\begin{equation}
    \Big(\frac{\eta_n}{n}\Big)^2 a^2_{n+1}\mu^2_{n+1}\EE( t_{n+1}^2 \mid \mathcal{F}_n) = O \Big(\frac{1}{n^2}\Big)
\end{equation}
and 
\begin{equation}
    \Big(\frac{1}{n}\Big)^2 \left( 1-p a_{n+1}\eta_{n+1}\mu_{n+1}\right)^2\EE( t_{n+1}^2 \mid \mathcal{F}_n) = O \Big(\frac{1}{n^2}\Big).
\end{equation}
Hence, we have almost surely
\begin{equation}
   \sum_{n\geq 1} \Big(\frac{\eta_n}{n}\Big)^2 \EE((\Delta M_{n+1})^2 \mid \mathcal{F}_n) <\infty \quad\text{and}\quad \sum_{n\geq 1} \Big(\frac{1}{n}\Big)^2 \EE((\Delta N_{n+1})^2 \mid \mathcal{F}_n) <\infty.
\end{equation}
Next observe by the proof of Corollary \ref{corollary.AsymptoticRates} we know that $\eta_n^{-1} \sim c a_n \mu_n$ for some positive constant $c$. Further, the sequence $(a_n \mu_n)$ is regularly varying of index $\alpha(1-p) -1$. Trivially, it holds that $\rho=\alpha(1-p)-1>-1=\delta$. By Proposition \ref{prop:equiv-RV} we thus know that $(\frac{n}{\eta_n})$ is eventually increasing.

Then, (2.17) from \cite[Theorem 2.18]{HallHeyde} ensures that
\begin{equation}
   \lim_{n\to\infty} \frac{\eta_n M_n}{n} =0\quad\text{and}\quad \lim_{n\to\infty}  \frac{N_n}{n} =0
\end{equation}
and we conclude from the definition of $(N_n)$ that 
\begin{align}
    \lim_{n \to \infty} \left( \frac{S_n + p \eta_n M_n}{n} \right) = 0 \qquad \text{a.s.} 
\end{align}
which achieves the proof.

\end{proof}

\subsection{Proof of Theorem \ref{thm:MainResult}}
Recall that we are working with the two-dimensional martingale $(\cM_n)$ defined by 
\begin{align}
    \cM_n = \binom{N_n}{M_n},
\end{align}
where $(M_n)$ and $(N_n)$ are the two square-integrable martingales introduced in Lemma \ref{lemma:Martingales}. By Corollary \ref{corollary:QuadVarMnNn} the main difficulty we face is that the predictable quadratic variation processes of $(M_n)$ and $(N_n)$ increase to infinity at two different rates. Hence we will require a matrix normalisation technique in order to establish the asymptotic behaviour of our elephant random walk. 

To simplify the proofs, they are provided here under the assumption that the steps are bounded, i.e. $\|X_k\|_\infty < \infty$ for any $k\geq1$. This assumption can be lifted through a truncation argument detailed in Appendix \ref{apx:bound}.

\begin{lem}
\label{lem:Vn-cMn}
    Let $(V_n)$ be the sequence of positive definite diagonal matrices of order $2$ given by 
    \begin{align}
    \label{eq:def-Vn}
        V_n = \frac{1}{\sqrt{n}} \begin{pmatrix}
            1 & 0  \\ 
            0 & p \eta_n
        \end{pmatrix},
    \end{align}
    then $\|V_n\|_\infty$ converges to zero as $n$ tends to infinity.
    
    Further, let $v = \binom{1}{1}$ such that 
    \begin{align}
        v^T V_n \mathcal{M}_n = \frac{S_n}{\sqrt{n}}.
    \end{align}
    The quadratic variation $\langle \mathcal{M} \rangle_n$ of $( \mathcal{M}_n)$ satisfies in the diffusive regime (i.e. $p < \frac{2 \alpha -1}{2 \alpha}$), 
    \begin{align}
        \lim_{n \to \infty} V_n \langle \mathcal{M} \rangle_n V_n^T = V \qquad \text{a.s.}
    \end{align}
    where the matrix $V$ is given by 
    \begin{align} \label{matrix:AsymptoticMatrixV}
        V= \begin{pmatrix}
             \frac{(1-\alpha)^2}{(1-\alpha(1-p))^2}& \frac{p(1-\alpha)}{(1-p)(1-\alpha(1-p))^2}  \\
             \frac{p(1-\alpha)}{(1-p)(1-\alpha(1-p))^2} & \frac{p^2 \alpha^2}{(1-\alpha(1-p))^2 (2 \alpha(1-p)-1)} 
         \end{pmatrix}
    \end{align}
\end{lem}

\begin{proof} For two sequence $(u_n)$ and $(v_n)$ we say that $u_n \propto v_n$ if there exists a constant $C$ such that $u_n \sim C v_n$. By Corollary \ref{corollary.AsymptoticRates} we have that 
    \begin{align}
        \eta_n \propto \frac{1}{a_n \mu_n} 
    \end{align}
    and the latter is a regularly varying sequence of index $\rho=1-\alpha(1-p)$. Further, it holds that $\rho < \delta=1/2$, because $p < \frac{2\alpha-1}{2\alpha}.$ It then follows that $(n^{-1/2} \eta_n)$ is eventually decreasing. Hence it follows that indeed $\|V_n\|_\infty \to 0$ as $n \to \infty$.

    Note that by Lemma \ref{lemma:QuadraticVariationProcess} and Lemma \ref{lem:RequiredConv}, we have for large enough $n$
    \begin{align}
        \langle \mathcal{M} \rangle_n = \mathcal{A}_n - \xi_n,
    \end{align}
    where 
    \begin{align}
        \mathcal{A}_n =  \sum_{k=0}^{n-1} \begin{pmatrix} \left( 1-p a_{k+1}\eta_{k+1}\mu_{k+1}\right)^2 & a_{k+1}\mu_{k+1}-aa_{k+1}^2 \eta_{k+1}\mu_{k+1}^2 \\
    a_{k+1}\mu_{k+1}-aa_{k+1}^2 \eta_{k+1}\mu_{k+1}^2 & a_{k+1}^2 \mu_{k+1}^2
    \end{pmatrix}
    \end{align}
    and 
    \begin{align}
        \xi_n = \sum_{k=0}^{n-1} \frac{p^2}{\nu_k^2}Y_k^2 \begin{pmatrix} \left( 1-p a_{k+1}\eta_{k+1}\mu_{k+1}\right)^2 & a_{k+1}\mu_{k+1}-aa_{k+1}^2 \eta_{k+1}\mu_{k+1}^2 \\
    a_{k+1}\mu_{k+1}-aa_{k+1}^2 \eta_{k+1}\mu_{k+1}^2 & a_{k+1}^2 \mu_{k+1}^2
    \end{pmatrix}.
    \end{align}
    Thanks to Theorem \ref{thm:LLN}, we immediately have that $\xi_n = \orm(\cA_n)$ since this is true for each coefficient of the matrix. In particular $\xi_n$ is negligible as $n$ tends to infinity. Hence we only need to consider $V_n \mathcal{A}_n V_n$ and thanks to the asymptotic rates established in Corollary \ref{corollary.AsymptoticRates} we arrive at 
    \begin{align}
        \lim_{n \to \infty} V_n \langle \cM \rangle_n V_n^T = V \qquad \text{a.s.}
    \end{align}
    with $V$ given by \eqref{matrix:AsymptoticMatrixV}.
\end{proof}

\begin{coro}[(H.1) of Theorem \ref{theorem:fcltMartingale}] \label{corollary:H1}
    In the diffusive regime, the quadratic variation of $( \mathcal{M}_n)$ satisfies for all $t\geq0$, 
    \begin{align}
        \lim_{n \to \infty} V_n \langle \mathcal{M} \rangle_{\floor{nt}} V_n^T = V_t \qquad \text{a.s.}
    \end{align}
    where the matrix $V_t$ is given by 
\begin{equation}
\label{def-Vt}
V_t = \frac{1}{(1-\alpha(1-p))^2}\begin{pmatrix}
(1-\alpha)^2 t & \dfrac{p(1-\alpha)}{1-p}t^{\alpha(1-p)}
    \\[0.8em] \dfrac{p(1-\alpha)}{1-p} t^{\alpha(1-p)} & \dfrac{p^2\alpha^2}{2\alpha(1-p)-1} t^{2\alpha(1-p)-1}\end{pmatrix}.
\end{equation}
\end{coro}

\begin{lem}[(H.2) of Theorem \ref{theorem:fcltMartingale}: Lindeberg's condition] \label{lem:Lindeberg}
    For all $t \geq 0$ and $\epsilon > 0$
    \begin{align}
        \sum_{k=1}^{ \tau(nt)} \EE \left( \| V_n \Delta \cM_k \|^2 \mathbf{1}_{\{ \| V_n \Delta \cM_k\| > \epsilon\}} \mid \cF_{k-1} \right) \xrightarrow[n \to \infty]{\text{a.s.}} 0, 
    \end{align}
    where $\Delta \cM_n = \cM_n-\cM_{n-1}$.
\end{lem}
\begin{proof}
First of all, we have from \eqref{eq:def-Vn}, \eqref{eq:DeltaM} and \eqref{eq-bound-epsilon} that 
\begin{equation}
\label{eq:bound-Vn-eps2}
\| V_n \Delta \cM_k \|^2 
= \frac{1}{n}\big((1-pa_{k}\eta_k\mu_k)^2+(p\eta_na_k\mu_k)\big)^2 t_{k}^2.
\end{equation}
Since $t_{k+1} = X_{k+1} - \frac{p}{\nu_k}Y_k$, we immediately have, by the assumption that our underlying steps $\mathcal{X}$ are bounded a.s., that $\sup_{k} |t_k| \leq \| \mathcal{X}\|_\infty < \infty$, and this ensures that
\begin{equation}
\label{eq:E-bound-Vn-eps2}
\E\big[\| V_n \Delta \cM_k \|^4\big] \leq \frac{1}{n^2}\big((1-aa_{k}\eta_k\mu_k)^2+p^2\eta_n^2a_k^2\mu_k^2\big)^2 \| \mathcal{X}\|_\infty^2 .
\end{equation}
It follows from the regularly varying properties of the sequences $\eta_n$, $a_n$ and $\mu_n$, and the fact that $p <\frac{2\alpha-1}{2\alpha}$ that,
\begin{equation}
    \eta_n^{-2}\sum_{k=1}^{\floor{nt}} (a_k\mu_k)^2  = O(n), \quad
\eta_n^{-4}\sum_{k=1}^{\floor{nt}} (a_k\mu_k)^4  = O(n)
\end{equation}
Hence, we find that
\begin{equation}
\sum_{k=1}^{\floor{nt}}\E\big[\| V_n \Delta \cM_k \|^4\big] = O \left( \frac{1}{n} \right) \quad\text{a.s.}
\end{equation}
and we deduce that 
\begin{equation}
\sum_{k=1}^{\floor{nt}} \E\bigl[\|V_n \Delta \cM_k \|^2 \1_{\{\|V_n\Delta \cM_k \|>\eps\}}\bigl|\cF_{k-1}\bigr] \leq \frac{1}{\eps^2}\sum_{k=1}^{\floor{nt}} \E\bigl[\|V_n \Delta \cM_k \|^4] \xrightarrow[n \to \infty]{\text{a.s.}} 0.
\end{equation}
This concludes the proof.
\end{proof}
The next Lemma establishes how we can decompose the matrix $V_t$.

\begin{lem}[(H.3) of Theorem \ref{theorem:fcltMartingale}]
\label{lem:decomp-Vt}
   The matrix $V_t$ can be written as $V_t = t K_1 + t^{\alpha_2} K_2 + t^{\alpha_3} K_3$ where 
${\alpha_2} = \alpha(1-p)>0$ and $\alpha_3=2\alpha(1-p)-1>0$
as $p<\frac{2\alpha-1}{2\alpha}$, and the matrix are symmetric,
\begin{equation}
K_1 =\frac{(1-\alpha)^2}{(1-\alpha(1-p))^2}\begin{pmatrix}1 & 0 \\ 0 & 0\end{pmatrix}, \
K_2 = \frac{p(1-\alpha)}{(1-p)(1-\alpha(1-p))^2}\begin{pmatrix}0 & 1 \\ 1 & 0\end{pmatrix},
\end{equation}
\begin{equation}
K_3 =\dfrac{p^2\alpha^2}{(1-\alpha(1-p))^2(2\alpha(1-p)-1)}\begin{pmatrix}0 & 0 \\ 0 & 1\end{pmatrix}.
\end{equation}
\end{lem}

\begin{proof}
    This is an immediate consequence of Corollary \ref{corollary:H1} and Display \eqref{matrix:AsymptoticMatrixV}.
\end{proof}

\begin{lem} \label{lemma:PenUltimateConvergence}
We have the following convergence in the Skorokhod space of càdlàg functions $D([0,+\infty))$,  
\begin{equation}
\big(V_n \cM_{\floor{nt}}, \ {t \geq 0}\big) \Longrightarrow \big(\cG_t, \ {t \geq 0}\big)
\end{equation}
where $\cG=\big(\cG_t, \ {t \geq 0}\big)$ is a continuous $\R^2$-valued centered Gaussian process starting at 0 with covariance, for $0\leq s\leq t$,
\begin{equation}
\label{COV-W-FCLT}
\E(\cG_s\cG_t^T) = V_s.
\end{equation}
\end{lem}

\begin{proof}
    Thanks to Corollary \ref{corollary:H1} and Lemmas \ref{lem:Lindeberg}, \ref{lem:decomp-Vt}, the claim follows immediately with an appeal to Theorem \ref{theorem:fcltMartingale}.
\end{proof}

We are now in a position to give a proof of our main result.

\begin{proof}[Proof of Theorem \ref{thm:MainResult}]
    Thanks to Lemma \ref{lemma:PenUltimateConvergence} we have the distributional convergence in the sense of Skorokhod as $n$ tends to infinity
    \begin{equation} \label{conv:MultMartingale}
    \big(V_n \cM_{\floor{nt}}, \ {t \geq 0}\big) \Longrightarrow \big(\cG_t, \ {t \geq 0}\big)
    \end{equation}
        
    Further, thanks to \eqref{eq:MartingaleIdentity} we can use the fact that $S_{\lfloor nt \rfloor}$ is asymptotically equivalent to 
    \begin{align}
        N_{\lfloor nt \rfloor} + \frac{p \alpha}{1-\alpha(1-p)}t^{1- \alpha(1-p)} \eta_n M_{\lfloor nt \rfloor}
    \end{align}
    Let now $u_t = (1, t^{1- \alpha(1-p)})^T$, then by multiplying \eqref{conv:MultMartingale} by $u_t^T$ from the left we obtain that 
    \begin{align}
        \left( \frac{S_{\lfloor nt \rfloor}}{\sqrt{n}}, t \geq 0 \right) \implies (W_t, t \geq 0),
    \end{align}
    where $W_t = u_t^T \cG_t$. In order to fully characterise the Gaussian process $W = (W_t, t\geq 0)$ it suffices to compute its covariance function. With an appeal to Lemma \ref{lem:decomp-Vt} and Display \eqref{COV-W-FCLT} we obtain for $0 \leq s \leq t$
    \begin{align}
        \EE(W_s W_t) &= u_s^T \EE ( \cG_s \cG_t^T) u_t \\
        &= u_s^T V_s u_t \\
        &= u_s^T (s K_1 + s^{\alpha(1-p)}K_2 + s^{2 \alpha(1-p)-1} K_3) \\
        &= \frac{(1-\alpha)^2}{(1-\alpha(1-p))^2}s +  \frac{p(1-\alpha)}{(1-p)(1-\alpha(1-p))^2}s \\
         & \quad + \frac{p(1-\alpha)}{(1-p)(1-\alpha(1-p))^2}s^{\alpha(1-p)}t^{1- \alpha(1-p)} \\
        & \quad +\frac{p^2 \alpha^2}{(1-\alpha(1-p))^2(2\alpha(1-p)-1)}s^{2 \alpha(1-p)-1}(st)^{1- \alpha(1-p)} \\
        &= \left( \frac{(1- \alpha)^2(1-p) + p(1-\alpha)}{(1-p)(1-\alpha(1-p))^2} \right)s \\
        & \quad + \frac{p(1- \alpha)}{(1-p)(1- \alpha(1-p))^2} s \left( \frac{t}{s} \right)^{1- \alpha(1-p)} \\
        & \quad + \left( \frac{p^2\alpha^2}{(1-\alpha(1-p))^2(2\alpha(1-p)-1)} \right) s \left( \frac{t}{s} \right)^{1- \alpha(1-p)} \\
        &= \left( \frac{1- \alpha}{(1-p)(1-\alpha(1-p))} \right)s + \left( \frac{p(\alpha(2-p)-1)}{(1-p)(2 \alpha(1-p)-1)(1-\alpha(1-p)} \right)s \left( \frac{t}{s} \right)^{1-\alpha(1-p)}
    \end{align}
    This concludes the proof of Theorem \ref{thm:MainResult}.
\end{proof}

\appendix 

\section{Technical Lemmas}

We provide here some technical results that are useful for our study but not directly related to the proofs or the martingale approach.

\begin{lem} \label{lem:RequiredConv}
     It holds that 
    \begin{align}
        \lim_{n \to \infty} \frac{U_n}{\nu_n} = 1 \qquad \text{a.s.}
    \end{align}
\end{lem}

\begin{proof}
    Thanks to Theorem \ref{thm:LLN}, it readily follows that  
    \begin{align}
        \lim_{n \to \infty} \frac{Y_n}{\nu_n}=0 \qquad \text{a.s.}
    \end{align}
    
    By assumption, we require our steps to be centred and of variance one. If this is no longer the case we instead modify the process such that these conditions are satisfied again. For example:
    \begin{align}
        S_n = X_1^2 + \dots + X_n^2
    \end{align}
    is a step-reinforced random walk with steps distributed as $X^2$, see the robustness in equation \eqref{eq:RobustnessSRRW}. In order to apply the LLN, we then instead work with 
    \begin{align}
        \tilde{S}_n &= X_1^2 - \EE(X_1^2) + \dots + X_n^2- \EE(X_n^2) \\
        &= X_1^2 + \dots + X_n^2 - n \EE(X_1^2) \\
        &= X_1^2 + \dots + X_n^2 -n.
    \end{align}
    The above, by the robustness of SRRW detailed in \eqref{eq:RobustnessSRRW}, is again a step-reinforced random walk, this time centred and hence the LLN applies and yields 
    \begin{align}
        \lim_{n \to \infty} \frac{\tilde{S}_n}{n}= 0 \qquad \text{a.s.}
    \end{align}
    which is equivalent to 
    \begin{align}
        \frac{X_1^2 + \dots + X_n^2}{n} \xrightarrow[n \to \infty]{a.s.} 1.
    \end{align}
    The exact same argument now holds for the process $U_n$ which is just a modification of $Y_n$ where instead of working with steps $X$ we work with steps $X^2- \EE(X^2)$. It then follows that
    \begin{align}
        \lim_{n \to \infty} \frac{\sum_{k=1}^n \mu_k (X_k^2 - \EE(X_k^2))}{\nu_n} & = \lim_{n \to \infty} \frac{\sum_{k=1}^n \mu_k X_k^2 - \sum_{k=1}^n \mu_k\EE(X_k^2) }{\nu_n} \\
        & = \lim_{n \to \infty} \frac{\sum_{k=1}^n \mu_k X_k^2 - \sum_{k=1}^n \mu_k }{\nu_n} \\
        & = \lim_{n \to \infty} \frac{\sum_{k=1}^n \mu_k X_k^2 - \nu_n}{\nu_n}  \\
        &= 0, \qquad \text{a.s.}
    \end{align}
    Or, equivalently, 
    \begin{align}
        \lim_{n \to \infty} \frac{U_n}{\nu_n}=1 \qquad \text{a.s.}
    \end{align}
\end{proof}

In light of Lemma \ref{lem:RequiredConv}, we see that the asymptotic behaviour of $( \tilde{w}_n)$, $(\tilde{z}_n)$ defined in Eq. \eqref{eq:def-tilde-w} and \eqref{eq:def-tilde-z} is fully characterised by 
\begin{align}
    w_n = \sum_{k=1}^n (a_k \mu_k)^2 
\end{align}
and 
\begin{align}
    z_n = \sum_{k=1}^n (1-pa_k \eta_k \mu_k)^2
\end{align}
respectively.

We now discuss the relevant asymptotic rates more closely. 

\begin{lem} \label{lemma:Asmptotics_an}
    The sequence $(a_n)$ is regularly varying of index $-p \alpha$.
\end{lem}

\begin{proof}
    By assumption, we have that $(\nu_n)$ is regularly varying of index $\alpha >0$, it follows that
    \begin{align}
        n \left(1 - \frac{a_{n-1}}{a_n} \right) &=  n \left(1 - \gamma_{n-1} \right) \\
        &= -p n \left( \frac{\nu_n}{\nu_{n-1}}-1 \right) \\
        &= -pn \left( \frac{\alpha}{n} + o \left( \frac{1}{n}\right) \right) \\
        &= - p \alpha + o(1).
    \end{align} 
    By Theorem \ref{theorem:RVSCharacterisation} the claim follows.
\end{proof}

\begin{coro} \label{corollary.AsymptoticRates}
    We record the following asymptotics:
    \begin{enumerate}
        \item $\displaystyle
            \lim_{n \to \infty} a_n \mu_n \eta_n = \frac{\alpha}{1- \alpha(1-p)}.$
        \item 
        $\displaystyle
        \lim_{n \to \infty} \frac{1}{n (a_n \mu_n)^2} w_n = \frac{1}{2 \alpha(1-p)-1}.$
        \item 
        $\displaystyle
            \lim_{n \to \infty} \frac{1}{n} z_n = \frac{(1-p)^2(1-\alpha)^2}{(1-\alpha(1-p))^2}.$
    \end{enumerate}
    
\end{coro}

\begin{proof}
    We proof each statement separately.
    \begin{enumerate}
    \item We have 
    \begin{align}
        \eta_n = \sum_{k=1}^{n-1} \frac{1}{a_k \nu_k}.
    \end{align}
    By definition, the sequence $(\nu_n)$ is regularly varying of index $\alpha$ and by Lemma \ref{lemma:Asmptotics_an} the sequence $(a_n)$ is regularly varying of index $- p \alpha$. It follows that $(a_n\nu_n)^{-1}$ is regularly varying of index $- \alpha(1-p)$ and it holds that $-\alpha(1-p) > -1$ for $\alpha < \frac{1}{1-p}$ or, equivalently, for $p > \frac{\alpha-1}{\alpha}$. By Theorem \ref{thm:RegVarCharacterisation} it follows that 
    \begin{align}
        \lim_{n \to \infty} \frac{a_n \nu_n}{n} \eta_n = \frac{1}{1 - \alpha(1-p)}.
    \end{align}
    Further, as $\nu_n \sim  \frac{n}{\alpha} \mu_n$, the claim follows.
        \item Recall that 
    \begin{align}
        w_n = \sum_{k=1}^n (a_k \mu_k)^2
    \end{align}
    and the sequence $(a_n)$ is regularly varying of index $-p \alpha$, whereas the sequence $(\mu_n)$ is regularly varying of index $\alpha-1$. Hence $(a_n \mu_n)^2$ is regularly varying of index $2 \alpha(1-p)-2$. We obverse that $2 \alpha(1-p) -2 > -1$ because $p< \frac{2 \alpha-1}{2 \alpha}$ and hence, by Theorem \ref{thm:RegVarCharacterisation}, it follows that 
    \begin{align}
        \lim_{n \to \infty} \frac{1}{n (a_n \mu_n)^2} w_n = \frac{1}{2\alpha(1-p)-1}.
    \end{align}

    \item Here we have 
    \begin{align}
        z_n = \sum_{k=1}^n (1- p a_k \eta_k \mu_k)^2.
    \end{align}
    It is then immediate from the first item that 
    \begin{align}
         \lim_{n \to \infty}\frac{1}{n} z_n = \frac{(1-p)^2(1-\alpha)^2}{(1-\alpha(1-p))^2}.
    \end{align}
    \end{enumerate}
    This concludes the proof.
\end{proof}
The reduction argument relies on the following lemma taken from \cite{Processes}, that we state for the  reader's convenience: 
\begin{lem}[Lemma 3.31 in Chapter VI of \cite{Processes}] \label{lemma:reductionLemma} \mbox{}\\
Let $(Z^n)$ be a sequence of $d$-dimensional rcll (càdlàg) processes and suppose that 
\begin{equation*}
    \forall N >0, \quad \forall \epsilon >0 \quad \quad \lim_{n \rightarrow \infty} \mathbb{P}\left( \sup_{s \leq N} |Z_s^n| > \epsilon \right)=0.
\end{equation*}
If $(Y^n)$ is another sequence of $d$-dimensional rcll processes with $Y^n \Rightarrow Y$ in the sense of Skorokhod, then $Y^n + Z^n \Rightarrow Y$ in the sense of Skorokhod. 
\end{lem}

Finally, we will need the following lemma concerning convergence on metric spaces: 
\begin{lem} \label{lemma:theReductionLemma} Let $(E,d)$ be a metric space and consider  $(a_n^{(m)} \, : \, m,n \in \mathbb{N})$ a family of sequences, with  $a_n^{(m)} \in E$ for all $n, m \in \mathbb{N}$. Suppose further that the following conditions are satisfied:
\begin{enumerate}
    \item For each fixed $m$, $a_n^{(m)} {\rightarrow} a_\infty^{(m)}$ as $n \uparrow \infty$ for some element $a_\infty^{(m)} \in E$.
    \item $a_\infty^{(m)} {\rightarrow} a_{\infty}^{(\infty)}$ as  $m \uparrow \infty$, for some $a_\infty^{(\infty)} \in E$.
\end{enumerate}
Then, there exists a non-decreasing  subsequence $(b(n))_{n}$ with $b(n) \rightarrow  \infty$ as $n \uparrow \infty$,  for which the following convergence holds:
\begin{equation*}
    a_n^{(b(n))} {\rightarrow} a_\infty^{(\infty)} \quad \text{ as } n \uparrow \infty.
\end{equation*}
\end{lem}
\begin{proof}
Since the sequence $(a_\infty^{(m)})_m$ converges, we can find an increasing subsequence $m_1 \leq m_2 \leq \dots $ satisfying 
\begin{equation*}
    d(a_\infty^{(m_k)} , a_\infty^{(m_{k+1})}  ) \leq 2^{-k} \quad \quad \text{ for each } k \in \mathbb{N}.
\end{equation*}
Moreover, since for each fixed $m_k$ the corresponding sequence $(a_n^{(m_k)})_n$ converges, there exists a strictly increasing sequence $(n_k)_k$ satisfying that, for each $k$, 
\begin{align*}
     d(a_i^{(m_k)} , a_\infty^{(m_k)} ) \leq 2^{-k} \quad \quad \text{ for all } i \geq n_k.
\end{align*}
Now, we set
for $n < n_1$, $b(n) := m_1$ and for $k \geq 1$, 
$b(n) := m_k$ if $n_k\leq n < n_{k+1}$ and we claim $(a_n^{b(n)})_n$ is the desired sequence. Indeed, it suffices to observe that for $n_k \leq n < n_{k+1}$,
\begin{equation*}
    d(a_n^{(b(n))} , a_\infty^{(\infty)}) 
    = d(a_n^{(m_k)} , a_\infty) 
    \leq d(a_n^{(m_k)} , a_\infty^{(m_k)}) + d(a_\infty^{(m_k)}, a_\infty) \leq  2^{-k} + \sum_{i=k}^\infty 2^{-i}.
\end{equation*}
\end{proof}

\section{Truncation argument for removing the boundness assumption}
\label{apx:bound}

We have established our main result Theorem \ref{thm:MainResult} under the simplifying assumption that the underlying steps are bounded, that is $\|\mathcal{X}_k\|_\infty < \infty$ for all $k \in \mathbb{N}$. In this section we present an argument to lift this restriction, inspired by  \cite{BertoinNoise, BertenghiOrtiz}. As such, we only make the assumption that $\EE(\mathcal{X})=0$ and $0< \text{Var}(\mathcal{X})=\sigma^2 < \infty$. 

First, we require the following bound:

\begin{lem} Let $\frac{\alpha-1}{\alpha}<p < \frac{2\alpha -1}{2 \alpha}$, then we have the bound
    \begin{align*}
        \EE \left( \sup_{k \leq n} |S_k|^2 \right) \leq \sigma^2 \left(4 z_n + 4p^2 \eta_n^2 w_n + 8p \sqrt{w_n z_n \eta_n^2}\right).
    \end{align*}
\end{lem}

\begin{proof}
    Recall the decomposition from \eqref{eq:MartingaleIdentity}
\begin{align*}
    S_n = N_n + p \eta_n M_n.
\end{align*}
Since $(\eta_n)$ is an increasing function and $(M_n),(N_n)$ are martingales, $(S_n)$ is a submartingale.  

Thanks to Doob's martingale inequality and the Cauchy-Schwarz inequality, we then have 
\begin{align*}
    \EE\left( \sup_{k \leq n} |S_k|^2 \right) &\leq 4 \EE( |S_n|^2) \\
    &= 4 \left( \EE(N_n^2) + p^2 \eta_n^2 \EE(M_n^2) + 2p \eta_n \EE(N_n M_n) \right) \\
    & \leq 4\EE(N_n^2) + 4p \eta_n^2 \EE(M_n^2) + 8p  \sqrt{\eta_n^2\EE(M_n^2) \EE(N_n^2)}.
\end{align*}
Recall from Corollary \ref{corollary:QuadVarMnNn}, Display \ref{eq:def-tilde-w} and Display \ref{eq:def-tilde-z} that respectively,
\begin{align*}
  \EE(N_n^2)&=  \EE ( \langle N \rangle_n) \leq \sigma^2 \sum_{k=0}^{n} (1-pa_{k} \eta_{k} \mu_{k})^2 = \sigma^2 z_n  \\
  \EE(M_n^2)&=    \EE \left( \langle M \rangle_n \right) \leq \sigma^2 \sum_{k=0}^n a_k^2 \mu_k^2 = \sigma^2 w_n 
\end{align*}
Using these bounds the claim follows.
\end{proof}
\begin{coro} \label{cor:TruncationFinal}
    There exists a non-negative constant $C$ such that
    \begin{align*}
        \lim_{n \to \infty} \frac{1}{n} \EE \left( \sup_{k \leq n} |S_k|^2 \right) \leq C\sigma^2.
    \end{align*}
\end{coro}
\begin{proof}
    This is now an immediate consequence of Corollary \ref{corollary.AsymptoticRates}. Indeed we have 
    \begin{align*}
     \frac{1}{n} \EE \left( \sup_{k \leq n} |S_k|^2 \right) & \leq \sigma^2 \left(  4 \frac{1}{n}z_n + 4p^2 \frac{1}{n}w_n \eta_n^2 + 8p \sqrt{\frac{1}{n}z_n \times \frac{1}{n}w_z \eta_n^2 } \right) \\
     & \sim  \sigma^2 \left( c_1 + c_2 (a_n \mu_n \eta_n)^2 + c_3 \right) \\
     & \sim C \sigma^2.
    \end{align*}
\end{proof}
We now split each underlying step $X_i$ for $i \in \mathbb{N}$ as 
\begin{align*}
    X_i = X_i^{ \leq K} + X_i^{>K}
\end{align*}
where respectively,
\begin{align*}
    X_i^{\leq K} & := X_i \mathbf{1}_{\{|X_i| \leq K\}} - \EE\left( X_i \mathbf{1}_{\{|X_i| \leq K \}} \right) \\
    X_i^{>K} &:= X_i \mathbf{1}_{\{|X_i|>K\}}- \EE\left( X_i \mathbf{1}_{\{|X_i|>K\}} \right),
\end{align*}
yields a natural decomposition for $(S_n)$ in terms of two step-reinforced random walks 
\begin{align*}
    \tilde{S}_n = S_n^{\leq K} + S_n^{>K},
\end{align*}
where $(S_n^{\leq K}), (S_n^{>K})$ are step-reinforced versions with typical step centred and distributed respectively as
\begin{align*}
    X^{\leq K} = X \mathbf{1}_{\{|X| \leq K\}} - \EE\left( X \mathbf{1}_{\{|X| \leq K \}} \right)
\end{align*}
and
\begin{align*}
    X^{>K} = X \mathbf{1}_{\{|X|>K\}}- \EE\left( X \mathbf{1}_{\{|X|>K\}} \right).
\end{align*}
Moreover, $X^{\leq K}$ is centred with variance $\sigma_K^2$ and $\sigma_K^2 \to \sigma^2$ as $K \nearrow \infty$. Similarly $X^{>K}$ is centred and we denoted it's variance by $\varsigma_K^2$, which converges towards zero as $K \nearrow \infty$. We will also write the respective truncated random walks as 
\begin{align*}
    S_n^{\leq K} &= X_1^{\leq K} + \dots + X_n^{\leq K}, \\
    S_n^{>K} &= X_1^{>K} + \dots + X_n^{>K}.
\end{align*}
Note, that thanks to Theorem \ref{thm:MainResult} we know that
\begin{align*}
    \left( \frac{S_{\lfloor nt \rfloor}^{ \leq K}}{\sigma_K \sqrt{n}}, t \geq 0\right) \implies (W_t, t \geq 0),
\end{align*}
where $(W_t, t \geq 0)$ is the Gaussian process specified in Theorem \ref{thm:MainResult}.

In order to apply Lemma \ref{lemma:reductionLemma}, we need the following Lemma:
\begin{lem} \label{lemma:reductionStep}
    For any sequence $(K_n)$ increasing towards infinity, we have
    \begin{align*}
         \lim_{n \to \infty} \frac{1}{n} \EE \left( \sup_{k \leq nt} \left|S_k^{>K_n} \right|^2\right)=0.
    \end{align*}
\end{lem}
\begin{proof}
    Recall that we denoted by $\varsigma_K^2$ the variance of $X^{>K}$ and further that $\varsigma_K^2 \to 0$ as $K \nearrow \infty$. Thanks to Corollary \ref{cor:TruncationFinal} we know that there exists some non-negative constant $C$ such that  
    \begin{align*}
        \lim_{n \to \infty} \frac{1}{n} \EE \left( \sup_{k \leq nt} \left|S_k^{>K_n} \right|^2\right)  \leq C \lim_{n \to \infty} \varsigma_{K_n}^2 t =0.
    \end{align*}
\end{proof}
We can now apply Lemma \ref{lemma:reductionLemma} to the processes
\begin{align*}
    Y_t^n = \frac{S_{\lfloor nt \rfloor}^{\leq K_n}}{\sqrt{n}}, \qquad Z_t^n = \frac{S_{\lfloor nt \rfloor}^{> K_n}}{\sqrt{n}}, \quad t \geq 0.
\end{align*}
We see from Lemma \ref{lemma:reductionStep} and the Markov inequality that Lemma \ref{lemma:reductionLemma} applies to said process. It then follows by the decomposition
\begin{align*}
    n^{-1/2} S_{\lfloor nt \rfloor} = Y_t^n + Z_t^n \implies \sigma W(t), \quad t \geq 0, \qquad \text{as } n \to \infty.
\end{align*}
This shows that Theorem \ref{thm:MainResult} holds for general, possibly unbounded, steps as long as $\mathcal{X} \in L^2( \mathbb{P})$.

\section{A non-standard result on martingales}
The proof of our main result, Theorem \ref{thm:MainResult}, relies on a non-standard functional central limit theorem for multi-dimensional martingales. A simplified version of Theorem 1 part2) of Touati \cite{Touati} is as follows.

\begin{theorem} \label{theorem:fcltMartingale}
    Let $(\cM_n)$ be a locally square-integrable martingale of $\mathbb{R}^d$ adapted to a filtration $(\cF_n)$, with predictable quadractic variation $\langle \cM \rangle_n$. Let $(V_n)$ be a sequence of non-random square matrices of order $d$ such that $\| V_n\|$ decreases to $0$ as $n$ tends to infinity. Moreover, let $\tau : \mathbb{R}_+ \to \mathbb{R}_+$ be a non-decreasing function going to infinity at infinity. Assume that there exists a symmetric and positive semi-definite matrix $V_t$ that is deterministic and such that for all $t \geq 0$ 
    
    \begin{flalign}
        V_n \langle \cM \rangle_{ \tau (nt)} V_n^T \xrightarrow[ n \to \infty]{ \IP} V_t. \tag{H.1}
    \end{flalign}

    \flushleft Moreover, assume that Lindeberg's condition is satisfied, that is for all $t \geq 0$ and $\epsilon >0$, 
    \begin{align}
        \sum_{k=1}^{ \tau(nt)} \EE \left( \| V_n \Delta \cM_k \|^2 \mathbf{1}_{\{ \| V_n \Delta \cM_k\| > \epsilon\}} \mid \cF_{k-1} \right) \xrightarrow[n \to \infty]{ \IP} 0, \tag{H.2}
    \end{align}
    where $\Delta \cM_n = \cM_n-\cM_{n-1}$.
    
    Finally,  assume that for some $q \in \mathbb{N}^*$
    \begin{align}
        V_t = \sum_{j=1}^q t^{\alpha_j}K_j \tag{H.3}
    \end{align}
    where $\alpha_j >0$ and $K_j$ is a symmetric matrix.

    Then, we have the distributional convergence in the Skorokhod space $D([0, \infty))$ of right-continuous functions with left-hand limits, 
    \begin{align}
        \left( V_n \cM_{\tau (nt)}, t \geq 0 \right) \implies ( \cG_t, t \geq 0) \tag{A.1}
    \end{align}
    where $\cG = ( \cG_t, t \geq 0)$ is a continuous $\mathbb{R}^d$-valued centered Gaussian process starting at $0$ with covariance function given for $0 \leq s \leq t$ for, 
    \begin{align}
        \EE \left( \cG_s \cG_t^T \right) = V_s. \tag{A.2}
    \end{align}
\end{theorem}

\bibliographystyle{abbrvnat}
\bibliography{bib}

\begin{thebibliography}{35}
\providecommand{\natexlab}[1]{#1}
\providecommand{\url}[1]{\texttt{#1}}
\expandafter\ifx\csname urlstyle\endcsname\relax
  \providecommand{\doi}[1]{doi: #1}\else
  \providecommand{\doi}{doi: \begingroup \urlstyle{rm}\Url}\fi

\bibitem[Bai et~al.(2002)Bai, Hu, and Zhang]{BaiGaussian}
Z.-D. Bai, F.~Hu, and L.-X. Zhang.
\newblock Gaussian approximation theorems for urn models and their
  applications.
\newblock \emph{The Annals of Applied Probability}, 12\penalty0 (4):\penalty0
  1149--1173, 2002.

\bibitem[Baur(2020)]{BaurClass}
E.~Baur.
\newblock On a class of random walks with reinforced memory.
\newblock \emph{Journal of Statistical Physics}, 181\penalty0 (3):\penalty0
  772--802, 2020.
\newblock \doi{10.1007/s10955-020-02602-3}.

\bibitem[Baur and Bertoin(2016)]{BaurBertoin}
E.~Baur and J.~Bertoin.
\newblock Elephant random walks and their connection to {P}\'olya-type urns.
\newblock \emph{Phys. Rev. E}, 94, 2016.
\newblock \doi{10.1103/PhysRevE.94.052134}.

\bibitem[Bercu(2017)]{Bercu}
B.~Bercu.
\newblock A martingale approach for the elephant random walk.
\newblock \emph{Journal of Physics A: Mathematical and Theoretical},
  51\penalty0 (1), 2017.
\newblock \doi{10.1088/1751-8121/aa95a6}.

\bibitem[Bercu and Guevara(2022)]{BercuMRW}
B.~Bercu and V.~H.~V. Guevara.
\newblock Further results on the minimal random walk.
\newblock 55\penalty0 (41):\penalty0 415001, oct 2022.
\newblock \doi{10.1088/1751-8121/ac92ad}.

\bibitem[Bercu and Laulin(2020)]{BercuLaulinCenter}
B.~Bercu and L.~Laulin.
\newblock On the center of mass of the elephant random walk.
\newblock \emph{Stochastic Processes and their Applications}, 2020.
\newblock \doi{10.1016/j.spa.2020.11.004}.

\bibitem[Bertenghi(2022)]{Bertenghi}
M.~Bertenghi.
\newblock Functional limit theorems for the multi-dimensional elephant random
  walk.
\newblock \emph{Stochastic Models}, 38\penalty0 (1):\penalty0 37--50, 2022.
\newblock \doi{10.1080/15326349.2021.1971092}.

\bibitem[Bertenghi and Rosales-Ortiz(2021)]{BertenghiOrtiz}
M.~Bertenghi and A.~Rosales-Ortiz.
\newblock Joint invariance principles for random walks with positively and
  negatively reinforced steps.
\newblock \emph{arXiv preprint arXiv:2109.11298}, 2021.

\bibitem[Bertoin(2020)]{BertoinNoise}
J.~Bertoin.
\newblock Noise reinforcement for {L}évy processes.
\newblock \emph{Ann. Inst. H. Poincaré Probab. Statist.}, 56\penalty0
  (3):\penalty0 2236--2252, 08 2020.
\newblock \doi{10.1214/19-AIHP1037}.

\bibitem[Bertoin(2021)]{BertoinUniversality}
J.~Bertoin.
\newblock Universality of noise reinforced brownian motions.
\newblock \emph{In and out of equilibrium 3: Celebrating Vladas Sidoravicius},
  pages 147--161, 2021.

\bibitem[Bojanic and Seneta(1973)]{Bojanic73}
R.~Bojanic and E.~Seneta.
\newblock A unified theory of regularly varying sequences.
\newblock \emph{Mathematische Zeitschrift}, 134\penalty0 (2):\penalty0 91--106,
  1973.
\newblock \doi{10.1007/BF01214468}.
\newblock URL \url{https://doi.org/10.1007/BF01214468}.

\bibitem[Brockmann et~al.(2006)Brockmann, Hufnagel, and Geisel]{Brockmann}
D.~Brockmann, L.~Hufnagel, and T.~Geisel.
\newblock The scaling laws of human travel.
\newblock \emph{Nature}, 439\penalty0 (7075):\penalty0 462--465, 2006.
\newblock \doi{10.1038/nature04292}.

\bibitem[Bronstein et~al.(2009)Bronstein, Israel, Kepten, Mai, Shav-Tal,
  Barkai, and Garini]{Bronstein}
I.~Bronstein, Y.~Israel, E.~Kepten, S.~Mai, Y.~Shav-Tal, E.~Barkai, and
  Y.~Garini.
\newblock Transient anomalous diffusion of telomeres in the nucleus of
  mammalian cells.
\newblock \emph{Phys. Rev. Lett.}, 103:\penalty0 018102, Jul 2009.
\newblock \doi{10.1103/PhysRevLett.103.018102}.

\bibitem[Buldyrev et~al.(1994)Buldyrev, Goldberger, Havlin, Peng, and
  Stanley]{Buldy}
S.~V. Buldyrev, A.~L. Goldberger, S.~Havlin, C.-K. Peng, and H.~E. Stanley.
\newblock Fractals in biology and medicine: from dna to the heartbeat.
\newblock In \emph{Fractals in science}, pages 49--88. Springer, 1994.

\bibitem[Chen and Laulin(2023)]{chenlaulin2023}
J.~Chen and L.~Laulin.
\newblock Analysis of the smoothly amnesia-reinforced multidimensional elephant
  random walk.
\newblock \emph{Journal of Statistical Physics}, 190\penalty0 (158), 2023.
\newblock \doi{10.1007/s10955-020-02590-4}.

\bibitem[Coletti et~al.(2017{\natexlab{a}})Coletti, Gava, and
  Schütz]{Coletti2017}
C.~F. Coletti, R.~Gava, and G.~M. Schütz.
\newblock A strong invariance principle for the elephant random walk.
\newblock \emph{Journal of Statistical Mechanics: Theory and Experiment},
  2017\penalty0 (12):\penalty0 123207, dec 2017{\natexlab{a}}.
\newblock \doi{10.1088/1742-5468/aa9680}.

\bibitem[Coletti et~al.(2017{\natexlab{b}})Coletti, Gava, and
  Schütz]{GavaSchuetzColetti}
C.~F. Coletti, R.~Gava, and G.~M. Schütz.
\newblock Central limit theorem and related results for the elephant random
  walk.
\newblock \emph{Journal of Mathematical Physics}, 58, 2017{\natexlab{b}}.
\newblock \doi{10.1063/1.4983566}.

\bibitem[Galambos and Seneta(1973)]{Galambos73}
J.~Galambos and E.~Seneta.
\newblock Regularly varying sequences.
\newblock \emph{Proceedings of the American Mathematical Society}, 41\penalty0
  (1):\penalty0 110--116, 1973.
\newblock ISSN 00029939, 10886826.
\newblock URL \url{http://www.jstor.org/stable/2038824}.

\bibitem[Gonz{\'a}lez-Navarrete and Lambert(2018)]{Gonzales}
M.~Gonz{\'a}lez-Navarrete and R.~Lambert.
\newblock Non-markovian random walks with memory lapses.
\newblock \emph{Journal of Mathematical Physics}, 59\penalty0 (11):\penalty0
  113301, 2018.
\newblock \doi{10.1063/1.5033340}.

\bibitem[González-Navarrete(2020)]{GonzalesMult}
M.~González-Navarrete.
\newblock Multidimensional walks with random tendency.
\newblock \emph{Journal of Statistical Physics volume}, 181:\penalty0
  1138--1148, 2020.
\newblock \doi{10.1007/s10955-020-02621-0}.

\bibitem[Guevara and Su{\'a}rez()]{GuevaraMinimal}
V.~H.~V. Guevara and H.~C. Su{\'a}rez.
\newblock A strategy to improve learning via a minimal random walk.

\bibitem[Gut and Stadtm{\"u}ller(2022)]{GutERWGradIncrMemory}
A.~Gut and U.~Stadtm{\"u}ller.
\newblock The elephant random walk with gradually increasing memory.
\newblock \emph{Statistics \& Probability Letters}, 189:\penalty0 109598, 2022.

\bibitem[Heyde and Hall(1980)]{HallHeyde}
C.~C. Heyde and P.~Hall.
\newblock \emph{Martingale Limit Theory and its Application}.
\newblock Academic Press, 1980.
\newblock \doi{10.1016/C2013-0-10818-5}.

\bibitem[Jacod and Shiryaev(2003)]{Processes}
J.~Jacod and A.~N. Shiryaev.
\newblock \emph{Limit Theorems for Stochastic Processes}.
\newblock Springer, 2003.
\newblock \doi{10.1007/978-3-662-05265-5}.

\bibitem[Koch and Brady(1988)]{Koch}
D.~L. Koch and J.~F. Brady.
\newblock Anomalous diffusion in heterogeneous porous media.
\newblock \emph{The Physics of fluids}, 31\penalty0 (5):\penalty0 965--973,
  1988.
\newblock \doi{10.1063/1.866716}.

\bibitem[Kubota and Takei(2019)]{KubotaTakei}
N.~Kubota and M.~Takei.
\newblock Gaussian fluctuation for superdiffusive elephant random walks.
\newblock \emph{Journal of Statistical Physics 177}, pages 1157--1171, 2019.
\newblock \doi{10.1007/s10955-019-02414-0}.

\bibitem[K\"ursten(2016)]{Kuersten}
R.~K\"ursten.
\newblock Random recursive trees and the elephant random walk.
\newblock \emph{Phys. Rev. E}, 93:\penalty0 032111, Mar 2016.
\newblock \doi{10.1103/PhysRevE.93.032111}.

\bibitem[Laulin(2022{\natexlab{a}})]{LaulinAmnesia}
L.~Laulin.
\newblock Introducing smooth amnesia to the memory of the elephant random walk.
\newblock \emph{Electronic Communications in Probability}, 27:\penalty0 1--12,
  2022{\natexlab{a}}.

\bibitem[Laulin(2022{\natexlab{b}})]{LaulinReinforcedERW}
L.~Laulin.
\newblock New insights on the reinforced elephant random walk using a
  martingale approach.
\newblock \emph{Journal of Statistical Physics}, 186\penalty0 (1):\penalty0
  1--23, 2022{\natexlab{b}}.
\newblock \doi{10.1007/s10955-021-02834-x}.

\bibitem[Miyazaki and Takei(2020)]{TakeiMiyazaki}
T.~Miyazaki and M.~Takei.
\newblock Limit theorems for the ‘laziest’minimal random walk model of
  elephant type.
\newblock \emph{Journal of Statistical Physics}, 181\penalty0 (2):\penalty0
  587--602, 2020.
\newblock \doi{10.1007/s10955-020-02590-4}.

\bibitem[Sch\"utz and Trimper(2004)]{SchuetzTrimper}
G.~M. Sch\"utz and S.~Trimper.
\newblock Elephants can always remember: Exact long-range memory effects in a
  non-markovian random walk.
\newblock \emph{Phys. Rev. E}, 70, 2004.
\newblock \doi{10.1103/PhysRevE.70.045101}.

\bibitem[Touati(1992)]{Touati}
A.~Touati.
\newblock On the functional convergence in distribution of sequences of
  semimartingales to a mixture of brownian motions.
\newblock \emph{Theory of Probability \& Its Applications}, 36\penalty0
  (4):\penalty0 752--771, 1992.

\bibitem[Vazquez and Cruz-Suárez(2020)]{GuevaraERW}
V.~Vazquez and H.~Cruz-Suárez.
\newblock An elephant random walk based strategy for improving learning.
\newblock 04 2020.
\newblock \doi{10.13140/RG.2.2.10920.72960}.

\bibitem[Weigel et~al.(2011)Weigel, Simon, Tamkun, and Krapf]{Weigel}
A.~V. Weigel, B.~Simon, M.~M. Tamkun, and D.~Krapf.
\newblock Ergodic and nonergodic processes coexist in the plasma membrane as
  observed by single-molecule tracking.
\newblock \emph{Proceedings of the National Academy of Sciences}, 108\penalty0
  (16):\penalty0 6438--6443, 2011.
\newblock ISSN 0027-8424.
\newblock \doi{10.1073/pnas.1016325108}.

\bibitem[Wong et~al.(2004)Wong, Gardel, Reichman, Weeks, Valentine, Bausch, and
  Weitz]{Wong}
I.~Y. Wong, M.~L. Gardel, D.~R. Reichman, E.~R. Weeks, M.~T. Valentine, A.~R.
  Bausch, and D.~A. Weitz.
\newblock Anomalous diffusion probes microstructure dynamics of entangled
  f-actin networks.
\newblock \emph{Phys. Rev. Lett.}, 92:\penalty0 178101, Apr 2004.
\newblock \doi{10.1103/PhysRevLett.92.178101}.

\end{thebibliography}

\end{document}